\documentclass[letterpaper,12pt,reqno]{amsart}

\usepackage[margin=1in]{geometry}

\usepackage{graphicx,amsmath,amssymb,amsthm,paralist,color,tikz-cd}
\usepackage{mathrsfs}
\usepackage{mathtools}
\usepackage{cite}
\usepackage{enumitem}  
\usepackage{setspace}

\usepackage{tikz}

\usepackage[color=blue!30]{todonotes}
\reversemarginpar
\usepackage{hyperref}
\hypersetup{
    colorlinks   = true,
     citecolor    = black,
    linkcolor    = blue
}

\numberwithin{equation}{section}

\newtheorem{theorem}{Theorem}[section]
\newtheorem{lemma}[theorem]{Lemma}

\newtheorem{question}[theorem]{Question}

\newtheorem{corollary}[theorem]{Corollary}

\newtheorem{thmintro}{Theorem}

\theoremstyle{definition}
\newtheorem{definition}[theorem]{Definition}

\newtheorem{remark}[theorem]{Remark}

\newcommand{\eps}{\varepsilon}

\DeclareMathOperator{\sys}{sys}
\DeclareMathOperator{\Mod}{Mod}
\DeclareMathOperator{\dist}{dist}
\DeclareMathOperator{\im}{i}
\DeclareMathOperator{\metr}{metr}
\DeclareMathOperator{\topol}{top}
\DeclareMathOperator{\harm}{harm}
\DeclareMathOperator{\vol}{vol}

\title[Geometry in a fixed conformal class]{Geometry and entropies in a fixed conformal class on surfaces}

\author[T.Barthelm\'e]{Thomas Barthelm\'e}
\address{Department of Mathematics and Statistics, Queen's University, Kingston, ON}
\email{thomas.barthelme@queensu.ca}

\author[A.Erchenko]{Alena Erchenko}
\address{Mathematics Department, Stony Brook University, Stony Brook, NY}
\email{alena.erchenko@stonybrook.edu}

\date{}

\begin{document}

\begin{abstract}
We show the flexibility of the metric entropy and obtain additional restrictions on the topological entropy of geodesic flow on closed surfaces of negative Euler characteristic with smooth non-positively curved Riemannian metrics with fixed total area in a fixed conformal class. Moreover, we obtain a collar lemma, a thick-thin decomposition, and precompactness for the considered class of metrics. Also, we extend some of the results to metrics of fixed total area in a fixed conformal class with no focal points and with some integral bounds on the positive part of the Gaussian curvature.
\end{abstract}

\maketitle

\section{Introduction}

When $M$ is a fixed surface, there has been a long history of studying how the geometric or dynamical data (e.g., the Laplace spectrum, systole, entropies or Lyapunov exponents of the geodesic flow) varies when one varies the metric on $M$, possibly inside a particular class. 

In \cite{BarthelmeErchenko}, we studied these questions in a class of metrics that seemed to have been overlooked: the family of non-positively curved metrics within a fixed conformal class. In this article, we prove several conjectures made in \cite{BarthelmeErchenko}, as well as give a fairly complete, albeit coarse, picture of the geometry of  non-positively curved metrics within a fixed conformal class.

Since Gromov's famous systolic inequality \cite{Gromov}, there has been a lot of interest in upper bounds on the systole (see for instance \cite{Guth}). In general, there is no positive lower bound on the systole. However, we prove here that non-positively curved metrics in a fixed conformal class do admit such a lower bound. 
\begin{thmintro}[Theorem~\ref{theorem: bound systole} and Corollary~\ref{corollary: bound topological entropy}]\label{thmintro:systole_and_entropy_bounds}
Let $\sigma$ be a fixed hyperbolic metric on a closed surface $M$ of negative Euler characteristic. Let $A>0$ be fixed. There exist positive constants $C_1,C_2$ depending on the topology of $M$, the metric $\sigma$ and $A$ such that
\begin{equation*}
\inf\limits_{g\in[\sigma]^\leq_A}\sys(g) \geq C_1 \text{ and } \sup\limits_{g\in[\sigma]^\leq_A}h_{\mathrm{top}}(g) \leq C_2,
\end{equation*}
where $[\sigma]^{\leq}_A$ is the family of smooth non-positively curved Riemannian metrics on $M$ that are conformally equivalent to $\sigma$ and have total area $A$.
\end{thmintro}
The above result implies in particular Conjecture~1.2 of \cite{BarthelmeErchenko}. 
We further would like to emphasize the fact that the bounds $C_1$ and $C_2$ that we obtain are explicit (although far from optimal).

In fact, we will prove Theorem \ref{thmintro:systole_and_entropy_bounds} for a larger class of metrics: those with no focal points and with total positive curvature bounded above by a constant smaller than $2\pi$ (see Theorem \ref{theorem: No focal points bound systole} and Corollary \ref{corollary: bound topological entropy no focal}).

Theorem \ref{thmintro:systole_and_entropy_bounds}, together with the flexibility result proven in \cite{ErchenkoKatok}, shows that the topological entropy of the geodesic flow on $M$ for a non-positively curved metric with fixed total area somehow detects some information about a conformal class. On the other hand, we show that the metric entropy is still completely flexible in any conformal class, proving Conjecture~1.1 of \cite{BarthelmeErchenko}.

\begin{thmintro}[Theorem~\ref{flex_metric}]
Let $M$ be a closed surface of negative Euler characteristic and $\sigma$ be a hyperbolic metric on $M$. Suppose $A>0$. Then,
$$\inf_{g\in [\sigma]^<_A}h_{\metr}(g)=0.$$
\end{thmintro}
The key ingredient in our proof of the above theorem is a way to smooth a conical singularity of a metric while preserving its conformal class. This technique is obtained in Lemma~\ref{smoothing_lemma}.

While trying to understand if there are additional restrictions for entropies in a fixed conformal class, we actually obtain a better picture of the coarse geometry of non-positively curved metrics.

Recall that a hyperbolic surface $(M,\sigma)$ can be decomposed into thick parts that have a bounded geometry and thin parts that are homeomorphic to annuli (see \cite[Chapter D]{BenedettiPetronio}). We show that the thick-thin decomposition of a hyperbolic surface determines a thick-thin decomposition for non-positively curved metrics that are conformally equivalent to the hyperbolic surface.  

\begin{thmintro}(Theorem~\ref{thick_thin})
For every thick piece $Y$ of $(M,\sigma)$, for every $g\in[\sigma]^\leq_A$, and for every non-trivial non-peripheral piecewise-smooth simple closed curve $\alpha$ in $Y$, the $g$-length of the $g$-geodesic representative of $\alpha$ is comparable to the $\sigma$-length of the $\sigma$-geodesic representative of $\alpha$ up to a multiplicative constant that depends only on the topology of $M$, the metric $\sigma$ and $A$.
\end{thmintro}

In addition, there is a well-known collar lemma for hyperbolic surfaces, i.e., if there exists a short non-trivial simple closed geodesic then the transversal closed geodesics are long. The collar lemma was generalized for Riemannian metrics with a lower curvature bound in \cite{Buser}. Lemma~\ref{collar lemma} is an analogous result for non-positively curved metrics in a fixed conformal class.

\subsection{Compactification and a result of Reshetnyak}

In the 1950s, Yuri Reshetnyak studied metrics on the disk of bounded integral curvature in the sense of Alexandrov. One of his results, \cite[Theorem 7.3.1]{Reshetnyak}, gives a compactification criterion for such metrics of bounded integral curvature (for the uniform topology), in terms of the curvature measure. In \cite{Troyanov_compact}, Troyanov extended that result (but without providing the complete proof) to the setting of metrics of bounded integral curvature on a closed surface and inside a fixed conformal class (see \cite[Theorem 6.2]{Troyanov_compact}).

It is natural to expect that one could obtain our Theorem \ref{thmintro:systole_and_entropy_bounds} starting from Troyanov's version of Reshetnyak's Theorem. Indeed, if one can prove, using Reshetnyak's Theorem, that the non-positively curved metrics considered in Theorem \ref{thmintro:systole_and_entropy_bounds} are precompact, then it would be enough to prove continuity of the systole amongst metrics of bounded integral curvature with the uniform topology. However,  Reshetnyak's Theorem does not apply directly to our case. Instead of carefully stating his theorem (we refer the reader to \cite[Theorem 6.2]{Troyanov_compact} and \cite[Theorem 7.3.1]{Reshetnyak} for the precise statements), which would require definitions that we do not need here, we will just point out the differences in the case of non-positively curved metrics. 

The main issue is that our metrics are scaled differently from those of Reshetnyak:
Suppose that $(g_n)$ is a sequence of metrics in $[\sigma]^\leq_A$. Then, Reshetnyak's theorem implies that there exists a sequence of Riemannian metrics $h_n = e^{2u_n}\sigma$ such that a subsequence converges to a metric of bounded curvature $h_{\infty} = e^{2u_{\infty}}\sigma$. However, the metrics $g_n$ and $h_n$ differ by a constant, i.e., there exists $C_n\in \mathbb{R}$ such that $g_n = C_n h_n$. Now the problem is that there is no a priori control of the constants $C_n$, and one would have to prove that they stay bounded away from $0$ and $+\infty$. (Note that, as a corollary of Theorem \ref{thmintro:systole_and_entropy_bounds}, this sequence is indeed bounded, see Theorem \ref{thmintro: precompact} below.)

A trivial example illustrates best this difference of scaling: Consider $\sigma$ a hyperbolic metric and $g_n = \sigma/n$. Then, $g_n$ obviously does not converge but the sequence $h_n$ that Reshetnyak's Theorem applies to is $h_n = n g_n = \sigma$, which does indeed trivially converge. Obviously, such an example does not preserve the total area, however, it is not obvious that one cannot construct a sequence of metrics that is $\sigma/n$ on a small disk, but still has non-positive curvature and fixed total area. The hard part in order to use Reshetnyak compacity result would be to prove directly that such sequences do not arise.

Therefore, we believe that our direct proof of Theorem \ref{thmintro:systole_and_entropy_bounds} is actually simpler than trying to use Reshetnyak's Theorem. Moreover, our result is stronger than what one could obtain via compactness, since we have an explicit dependency for the bounds $C_1$ and $C_2$ of Theorem \ref{thmintro:systole_and_entropy_bounds} (see Theorem~\ref{theorem: bound systole}).

Note that as a corollary of Theorem \ref{thmintro:systole_and_entropy_bounds} and a result of Debin, \cite[Corollary 5]{Debin}, we do get precompactness in the uniform metric sense of the class of metrics we consider

\begin{thmintro}(Theorem~\ref{precompactness})\label{thmintro: precompact}
The set of metrics $[\sigma]^\leq_A$ is precompact in the uniform metric sense (see Definition~\ref{uniform metric}) with the limiting metrics having bounded integral curvature. 
\end{thmintro}

Despite Reshetnyak's compactness criterion not being directly useful to us, we believe that the gist of his result might still apply, i.e., that as long as the family of metrics in a fixed conformal class one considers is away from cusped Alexandrov surfaces, then the systole and entropy are bounded.
To make that question more precise, we need to introduce some notations. If $g$ is a Riemannian metric with conical singularities, we denote by $K^+_g$ the positive part of its curvature. We also write $\mu^+_g$ for the positive part of the curvature measure of $g$, i.e., $\mu^+_g = K^+_g d\mathrm{vol}_g$, where $d\mathrm{vol}_g$ is the area measure (which contain atoms at the conical points of $g$).

Now, a natural question is
\begin{question} \label{question_Reshetnyak}
Let $\sigma$ be a fixed hyperbolic metric on a closed surface $M$ of negative Euler characteristic. Let $A, C, \eps>0$ be fixed. We define $[\sigma]^{\eps, C}_A$ as the set of Riemannian metrics with conical singularities $g$, conformally equivalent to $\sigma$, of total area $A$, and such that, 
\begin{enumerate}
 \item the total positive curvature is bounded, $\mu^+(M) \leq C; $
 \item for all $x\in M$, there exists $\eta= \eta(x)>0$ (so $\eta$ is independent of the metric $g$) such that
 \[
 \mu^+_g\left(B_{\sigma}(x,\eta)\right) \leq 2\pi-\eps,
 \]
where $B_{\sigma}(x,\eta)$ is the ball of radius $\eta$ for the metric $\sigma$.
\end{enumerate}

Do there exist positive constants $C_1,C_2$ depending on the topology of $M$, the metric $\sigma$, $A$, $C$, $\eps$, and possibly the function $\eta$, such that
\begin{equation*}
\inf\limits_{g\in[\sigma]^{\eps,C}_A}\sys(g) \geq C_1 \text{ and } \sup\limits_{g\in[\sigma]^{\eps,C}_A}h_{\mathrm{vol}}(g) \leq C_2?
\end{equation*}
\end{question}

Our proof unfortunately does not work to prove quite that strong a result. In Section \ref{section: no focal points}, we extend our arguments to their natural limits: We need to assume that the metrics have no focal points (and bounded total positive curvature) for Lemma \ref{lemma:Rafi_No_focal_version} to hold, and we need to assume that the \emph{total} positive curvature is less than $2\pi-\eps$ (rather than the much weaker no-concentration condition as in (2) above), for our proof of Theorem \ref{theorem: No focal points bound systole} to work.


\subsection{Organization of the paper}
In Section~\ref{section:systole}, we prove a uniform lower bound on the length of a systole for the family of smooth non-positively curved Riemannian metrics with fixed total area in a fixed conformal class. Then, in Section \ref{section: no focal points} we extend those results to the setting of surfaces without focal points. In Section \ref{section_precompactness}, we show precompactness in the uniform metric sense of the considered metrics. We obtain a thick-thin decomposition in Section \ref{section:thick-thin}. The flexibility of metric entropy is proved in Section~\ref{section:metric entropy}. Some natural open questions are formulated in Section~\ref{section:questions}.

\subsection*{Acknowledgements}
We would like to thank Yair Minsky, Federico Rodriguez Hertz, and Dennis Sullivan for useful discussions and questions. We also thank Ian Frankel for pointing out Yuri Reshetnyak's work to us and Cl\'ement Debin for several comments that helped us better formulate Question~\ref{question_Reshetnyak}.

\section{Collar lemma}\label{section:systole}

Consider a Riemannian metric $g$ on a closed surface $M$ of Euler characteristic $\chi(M)<0$. Denote by $[g]$ the family of metrics conformally equivalent to $g$. Since all of our results apply trivially to any finite cover of $M$, we always assume $M$ to be orientable.


Let $\gamma$ be a simple closed curve on $M$ and $[\gamma]$ be a family of simple closed curves isotopic to $\gamma$. Denote by $l_g(\cdot)$ the $g$-length of a curve and $A_g(\cdot)$ the $g$-area of a set.

\begin{definition}\label{extremal_length}
The extremal length of $[\gamma]$ with respect to a Riemannian metric $g$ is
\begin{equation}\label{definition: extremal length}
E_g(\gamma) = \sup\limits_{g'\in [g]}\frac{\inf\limits_{\gamma'\in[\gamma]}l^2_{g'}(\gamma')}{A_{g'}(M)}.
\end{equation}
\end{definition}

Notice that $E_g(\cdot)$ depends only on the conformal class of $g$.

\begin{definition}\label{definition: modulus}
The modulus $\Mod_g(\mathcal A)$ of an annulus $\mathcal A$ on $(M, g)$ is the reciprocal of $E_{g}(\gamma)$, where $\gamma$ is a simple closed curve isotopic to a boundary curve of $\mathcal A$. 

Moreover, $\Mod_g(\mathcal A) = E_{g}(c)=\sup\limits_{g'\in [g]}\frac{\inf\limits_{c'\in[c]}l^2_{g'}(c')}{A_g'(\mathcal A)}$, where $c$ is any path connecting the boundaries of $\mathcal A$ and $[c]$ is the family of curves that connect the boundaries of $\mathcal A$ and isotopic to $c$. In particular, $\Mod_g(\cdot)$ depends only on the conformal class of $g$.
\end{definition}

Let $\gamma$ be a smooth curve that is a boundary of a set $S$ in $M$. We choose the sign of the geodesic curvature of $\gamma$ to be positive when the acceleration vector points into $S$. We denote  by $\kappa(\gamma)$ the integral of the geodesic curvature of $\gamma$.

In the following Lemma we establish a lower bound for the modulus over certain annuli. This result is based on \cite[Lemma 3.6]{Rafi} (which itself uses \cite[Theorem 4.5]{Minsky}), but adapted to our context of smooth non-positively curved Riemannian metrics. Our result is weaker than \cite[Lemma 3.6]{Rafi} because we only give a lower bound, instead of both lower and upper bounds, for the modulus, but this is all we will need in order to prove Theorem \ref{thmintro:systole_and_entropy_bounds}
\begin{lemma}\label{lemma: Rafi}
Let $g$ be a smooth non-positively curved Riemannian metric on a closed surface $M$ with $\chi(M)<0$. Let $\mathcal A$ be an annulus in $(M,g)$. Let $\gamma_0$ and $\gamma_1$ be its two boundary curves. Assume that $\gamma_0$ and $\gamma_1$ are both equidistant to a fixed geodesic. Suppose, moreover, that $\kappa(\gamma_0)\leq 0$.
Then,
\begin{align*}
&\Mod_g(\mathcal A) = \frac{\dist_g(\gamma_0, \gamma_1)}{l_{g}(\gamma_0)} \text{ if } \kappa(\gamma_0)=0 \text{ and } \mathcal A \text{ is a flat annulus}, \\
&\Mod_g(\mathcal A)\geq \frac{1}{-2\pi\chi(M)}\ln\left(1-2\pi\chi(M)\frac{\dist_g(\gamma_0, \gamma_1)}{l_{g}(\gamma_0)}\right) \text{ otherwise, }
\end{align*}
where $\dist_g(\cdot, \cdot)$ is the distance function on $(M,g)$
\end{lemma}

\begin{proof}
If $\kappa(\gamma_0)=0$ and $\mathcal A$ is a flat annulus, then the result is classical (see \cite[Chapter 4]{Ahlfors}).

Now, consider the level curves $\hat\gamma_r := \{p\in\mathcal A| \dist_g(p, \gamma_0) = r\}$. 

Since $g$ has non-positive curvature, the annulus $A$ is foliated by the level curves $\hat\gamma_r$, $0\leq r\leq \dist_g(\gamma_0, \gamma_1)$.
Furthermore, $\hat\gamma_0 = \gamma_0$ and $\hat\gamma_{\dist_g(\gamma_0, \gamma_1)} = \gamma_1$.

We define a scaling function $f$ by $f(r) := \frac{l_{g}(\hat\gamma_0)}{l_{g}(\hat\gamma_r)}$. Then, for any $r$, the curve $\hat\gamma_r$ has length, in the metric $f(r)g$, $l_{fg}(\hat\gamma_r) = l_{g}(\gamma_0)$.

Let $\mathcal A_r$ be the annulus in $\mathcal A$ that is bounded by $\gamma_0$ and $\hat\gamma_r$.

In order to bound from below the modulus of $\mathcal{A}$, we will use that it is the extremal length of paths connecting the two boundaries.

Consider $c$ a path from $\gamma_0$ to $\gamma_1$. Then, its $fg$-length satisfies
\begin{equation*}
l_{fg}(c) \geq \dist_{fg}(\gamma_0, \gamma_1) = \int\limits_0^{\dist_g(\gamma_0, \gamma_1)}f(r)dr = \int\limits_0^{\dist_g(\gamma_0, \gamma_1)}\frac{l_{g}(\gamma_0)}{l_{g}(\hat\gamma_r)}dr = l_{g}(\gamma_0)\int\limits_0^{\dist_g(\gamma_0, \gamma_1)}\frac{1}{l_{g}(\hat\gamma_r)}dr. 
\end{equation*}

Moreover,
\begin{align*}
\vol_{fg}(\mathcal A)  = \int\limits_0^{\dist_g(\gamma_0, \gamma_1)}f^2(r)l_{g}(\hat\gamma_r)dr = \int\limits_0^{\dist_g(\gamma_0, \gamma_1)}\frac{l^2_{g}(\gamma_0)}{l_{g}(\hat\gamma_r)}dr = l^2_{g}(\gamma_0)\int\limits_0^{\dist_g(\gamma_0, \gamma_1)}\frac{1}{l_{g}(\hat\gamma_r)}dr.
\end{align*}

Let $K_{g}(\cdot)$ denote the Gaussian curvature function on $(M,g)$. Applying Gauss-Bonnet Theorem to $(M,g)$ and $(\mathcal A_r,g)$ gives that $\int_{M}K_{g}(x) dA_{g}(x)=2\pi\chi(M)$ and $\int_{\mathcal A_r}K_{g} dA_{g}+\kappa(\hat\gamma_r)-\kappa(\gamma_0) = 0$. Now, our assumptions are that $\chi(M)<0$, $\kappa(\gamma_0)\leq 0$, and $K_g(x)\leq 0$ for any $x\in (M,g)$. Thus, we have $\kappa(\hat\gamma_r)\leq -2\pi\chi(M)$.

Combining the previous inequality with the fact that $\frac{d}{dr}l_{g}(\hat\gamma_r) = \kappa(\hat\gamma_r)$, we obtain $l_{g}(\hat\gamma_r)\leq -2\pi\chi(M)r+l_{g}(\gamma_0)$.  

Therefore, using Definition~\ref{definition: modulus}, we have
\begin{align*}
\Mod_g(\mathcal A)&\geq\frac{\dist^2_{fg}(\gamma_0, \gamma_1)}{\vol_{fg}(\mathcal A)} = \int\limits_0^{\dist_g(\gamma_0, \gamma_1)}\frac{1}{l_{g}(\hat\gamma_r)}dr\geq \int\limits_0^{\dist_g(\gamma_0, \gamma_1)}\frac{1}{l_{g}(\gamma_0)-2\pi\chi(M)r}dr \\
& = \frac{1}{-2\pi\chi(M)}\ln\left(1-2\pi\chi(M)\frac{\dist_g(\gamma_0, \gamma_1)}{l_{g}(\gamma_0)}\right).
\end{align*}
\end{proof}

\begin{theorem}\label{theorem: bound systole}
Let $M$ be a closed surface of Euler characteristic $\chi(M)<0$ and $\sigma$ be a hyperbolic metric on $M$. Let $A>0$. Then, there exists a positive constant $C = C(\sigma, A)$ such that
\begin{equation*}
\inf\limits_{g\in[\sigma]^{\leq}_A}\sys(g)\geq C,
\end{equation*}
where $[\sigma]^{\leq}_A$ is a family of smooth non-positively curved metrics conformally equivalent to $\sigma$ with total area $A$ and $\sys(g)$ is the length of the shortest simple nontrivial closed geodesic for the metric $g$. 

Moreover, the constant $C$ is explicitly given by $C := \sqrt{\frac{A}{R}}$, where
\begin{align*}
R &= \pi(\chi^2(M)-\chi(M))\hat R^2+\frac{2}{\pi}(1-\pi\chi(M))\hat R+\frac{1}{\pi}, \\
\hat R &= \frac{e^{2\pi E(\chi^2(M)-3\chi(M))}-1}{-\pi\chi(M)}, \text{ and}\\
E&=E(\sigma)=\sup\limits_{\mathcal A - \text{ annulus }\subset M}\Mod_{\sigma}(\mathcal A).
\end{align*}
\end{theorem}

Theorem 1.2 in \cite{Sabourau} states that there exists a constant $C>0$ such that for every Riemannian metric $g$ on $M$ we have $$\sys(g)h_{\vol}(g)\leq C,$$
where $h_{\vol}(g)$ is the volume entropy on $(M,g)$. Moreover, the topological entropy coincides with the volume entropy for non-positively curved metrics \cite[Theorem 2]{Manning_volume}. Therefore, we obtain the following corollary of Theorem~\ref{theorem: bound systole}, which, in particular, proves Conjecture 1.2 of \cite{BarthelmeErchenko}.

\begin{corollary}\label{corollary: bound topological entropy}
Let $M$ be a closed surface with $\chi(M)<0$, and $\sigma$ be a hyperbolic metric on $M$. Let $A>0$. Then, there exists a positive constant $B = B(\sigma, A)$ such that
\begin{equation*}
\sup\limits_{g\in[\sigma]^{\leq}_A}h_{\topol}(g)\leq B,
\end{equation*}
where $h_{\topol}(g)$ is the topological entropy of the geodesic flow on $M$ with respect to the metric $g$.
\end{corollary}

In order to prove Theorem \ref{theorem: bound systole}, we will follow the proof of \cite[Lemma 4.1]{Rafi} while adapting it to our setting.
\begin{proof}[Proof of Theorem~\ref{theorem: bound systole}]
Let $g$ be a smooth non-positively curved Riemannian metric on $M$ with total area $A$ and conformally equivalent to $\sigma$. Denote by $\gamma$ the shortest simple closed nontrivial geodesic for $g$. Let $N_r$ be the open $r$-neighborhood of $\gamma$ in $(M,g)$. Then, we define $Z_r$ to be the union of $N_r$ and all components of $M\setminus N_r$ that are disks. Notice that since $\gamma$ is a smooth curve, $\partial Z_r$ is a piecewise smooth curve with at most isolated singularities that appear where the topology of $N_r$ changes (see, for example, \cite{cutlocus}). Let $\kappa(\partial Z_r)$ be the integral of the geodesic curvatures of the boundary components of $Z_r$ (including weights for the isolated singularities). Recall that the sign of $\kappa(\partial Z_r)$ is chosen with respect to an inward pointing normal vector of $Z_r$. Hence, by convexity of the distance function, we have that $\kappa(\partial Z_r)\geq 0$.
We denote by $K_g(\cdot)$ the Gaussian curvature function of $(M,g)$.

By Gauss-Bonnet Theorem applied to $(M,g)$ and $(Z_r,g)$, we have
\begin{align}
\int_{Z_r}K_g(x)dA_g(x)+\kappa(\partial Z_r) &= 2\pi\chi(Z_r) \quad \text{ and }\label{curvature}\\
\int_MK_g(x)dA_g(x)& = 2\pi\chi(M), \quad \text{respectively.} \nonumber
\end{align}
By Equation \eqref{curvature}, we have, for any $r$ 
\begin{equation}\label{bound_geodesic_curvature}
\kappa(\partial Z_r)\leq -2\pi\chi(M)=K
\end{equation}
as $\chi(Z_r)\leq 0$ and $K_g(x)\leq 0$ for any $x\in (M,g)$.

Using the first variation formulas for arc length and area (see, for example, \cite{Chavel}), the functions $l_g(\partial Z_r)$ and $A_g(Z_r)$ are differentiable functions of $r$ everywhere except for finitely many $r$, where we add a disk. For those $r$ where the functions are differentiable we have $\frac{d}{dr}l_g(\partial Z_r)=\kappa(\partial Z_r)$ and $\frac{d}{dr}A_g(Z_r)=l_g(\partial Z_r)$. Define $I_r$ to be the set of all indexes $u$ such that $Z_{r_u}=N_{r_u}\cup D_u$ where $D_u$ is the union of disjoint disks and $r_u\leq r$. Let $c_u = l_g(\partial D_u)$, i.e., it is the $g$-length of the boundary of $D_u$. As a result, we obtain that
\begin{align}\label{length_area_nbh}
&l_g(\partial Z_r)-l_g(\partial Z_0) = \int\limits_0^r\kappa(\partial Z_\tau) d\tau - \sum\limits_{u\in I_r}c_u,\\
&A_g(Z_r) - A_g(Z_0) = \int\limits_0^rl_g(\partial Z_{\tau}) d\tau + \sum\limits_{u\in I_r}A_g(D_u)\nonumber.
\end{align}
Furthermore, by the isoperimetric inequality (see \cite{Izmestiev}), we have
\begin{equation}\label{isoperimetric}
A_g(D_u)\leq \frac{c_u^2}{4\pi}.
\end{equation}
Therefore, combining Equations \eqref{bound_geodesic_curvature}, \eqref{length_area_nbh}, \eqref{isoperimetric} and the facts that $l_g(\partial Z_0)=2l_g(\gamma)$ and $A_g(Z_0)=0$, the following inequalities hold:
\begin{align}
&l_g(\partial Z_r)\leq Kr+2l_g(\gamma), \label{length}\\
&\sum\limits_{u\in I_r}c_u\leq Kr+2l_g(\gamma),
\end{align} 
and 
\begin{align}\label{upper bound on area}
A_g(Z_r)&\leq \int\limits_0^r(K\tau+2l_g(\gamma))d\tau+\frac{1}{4\pi}\left(\sum\limits_{u\in I_r}c_u^2\right)\nonumber \\&\leq \frac{Kr^2}{2}+2l_g(\gamma)r+\frac{1}{4\pi}\left(\sum\limits_{u\in I_r}c_u\right)^2\\
&\leq \frac{Kr^2}{2}+2l_g(\gamma)r+\frac{1}{4\pi}\left(Kr+2l_g(\gamma)\right)^2\nonumber.
\end{align}

Let $r_0=0$ and $\{r_i\}_{i=1}^s$ be the increasing sequence of values of $r$ where the topology of $Z_r$ changes. Notice that $s\leq 2-\chi(M)$ because $g$  has non-positive curvature and $Z_s = M$. By the definition of $\{r_i\}_{i=0}^s$, $Z_{r_{i+1}}\setminus Z_{r_{i}}$ is a union of annuli with monotonically curved equidistant boundary curves for every $i=0, \dots, s-1$. 
Moreover, for each annuli in $Z_{r_{i+1}}\setminus Z_{r_{i}}$ we have that the distance between its boundaries is $r_{i+1}-r_i$ (by construction) and the length of the shorter boundary is at most $Kr_i+2l_g(\gamma)$ (see Equation \eqref{length}).

Due to the choice of sign for the definition of the geodesic curvature, notice that each annuli in $Z_{r_{i+1}}\setminus Z_{r_{i}}$ as one boundary $\alpha_i$ such that $\kappa(\alpha_i)\leq 0$ and the other, $\alpha_{i+1}$ such that $\kappa(\alpha_{i}+1)\geq 0$. Thus we can apply Lemma~\ref{lemma: Rafi} to each annuli. The lemma yields
\begin{align*}
r_{i+1}-r_i&\leq \frac{e^{EK}-1}{K}(Kr_i+2l_g(\gamma)),
\end{align*}
where $E= \sup\limits_{\mathcal A - \text{ annulus }\subset M}\Mod_{\sigma}(\mathcal A)$.
Therefore, for any $i=0, \dots, s-1$, we have
\begin{equation*}\label{upper bound on length}
r_{i+1}\leq Pr_i+Ql_g(\gamma),
\end{equation*}
where $P = e^{EK}$ and $Q = 2\frac{e^{EK}-1}{K}$.
By induction, we get
\begin{equation}
r_s 
\leq P^{s+1}r_0+Ql_g(\gamma)\sum\limits_{i=0}^s P^i\leq Q\frac{P^{s+1}-1}{P-1}l_g(\gamma) \label{upper bound on distance}
\end{equation} 
Since $Z_s = M$, Equation \eqref{upper bound on area} together with Equation \eqref{upper bound on distance}, gives
\begin{equation*}
A = A_g(Z_s)\leq \frac{Kr_s^2}{2}+2l_g(\gamma)r_s+\frac{1}{4\pi}\left(Kr_s+2l_g(\gamma)\right)^2\leq R l^2_g(\gamma),
\end{equation*}
where 
\begin{align*}
&R = \pi(\chi^2(M)-\chi(M))\hat R^2+\frac{2}{\pi}(1-\pi\chi(M))\hat R+\frac{1}{\pi} \qquad \text{ and }\\
&\hat R = 2\frac{e^{-2\pi E\chi(M)}-1}{-2\pi\chi(M)}\cdot\frac{e^{-2\pi E\chi(M)(s+1)}-1}{e^{-2\pi E\chi(M)}-1} = \frac{e^{-2\pi E\chi(M)(s+1)}-1}{-\pi\chi(M)}.
\end{align*}
Therefore, we have
\begin{equation*}
\sys(g) = l_g(\gamma)\geq\sqrt{\frac{A}{R}}.
\end{equation*}
Using the fact that $s\leq 2-\chi(M)$, we prove the theorem.
\end{proof}

Following the model of \cite[Lemma 4.1]{Rafi}, with the same adaptations as the ones made in our proof of Theorem \ref{theorem: bound systole} above, we get a collar lemma.

\begin{lemma}[Collar lemma]\label{collar lemma}
Consider a closed surface $M$ of negative Euler characteristic. For every $L>0$, there exists a constant $D_L$ such that the following holds. Let $\alpha$ and $\beta$ be any two simple closed curves in $M$ that intersect non-trivially. Let $g$ be a smooth non-positively curved Riemannian metric on $M$ that is conformally equivalent to the hyperbolic metric $\sigma$. If $l_{\sigma}(\beta_{\sigma})\leq L$, then we have
$$D_L l_g(\alpha_g)\geq l_g(\beta_g),$$
where $\beta_{\sigma}$ is the $\sigma$-geodesic representative of $\beta$ and $\alpha_g$ and $\beta_g$ are the $g$-geodesic representatives of $\alpha$ and $\beta$, respectively. 
\end{lemma}
\section{Extension to metrics with no focal points} \label{section: no focal points}

In this section, we extend Lemma \ref{lemma: Rafi} and Theorem \ref{theorem: bound systole} to the setting of surfaces with no focal points.

The main interest of this extension is that it shows the limits of our proof, as well as the place where the assumption about no concentration of the positive curvature made in Question \ref{question_Reshetnyak} is necessary.

\begin{theorem}\label{theorem: No focal points bound systole}
Let $M$ be a closed surface of Euler characteristic $\chi(M)<0$ and $\sigma$ be a hyperbolic metric on $M$. Let $A>0$ and $\eps>0$. Then, there exists a positive constant $C = C(\sigma, A, \eps)$ such that the following holds:

For every Riemannian metric $g$ with no focal points in the conformal class of $\sigma$, of total area $A$, and such that,
\[
\int_{M} K^+_g d\mathrm{vol}_g <2\pi-\eps,
\]
where $K^+_g$ is the positive part of the Gaussian curvature of $g$, we have 
\[
\sys(g) > C.
\]
\end{theorem}

As in Theorem \ref{theorem: bound systole}, the bound $C$ can be made completely explicit.

Before going on to the proof of Theorem \ref{theorem: No focal points bound systole}, we state and prove the extension of Lemma \ref{lemma: Rafi}.  Note that it is for this result that we need to assume that the metric has no focal points. We recall that a surface has no focal points if and only if, in its universal cover, every point admits a unique orthogonal projection onto any geodesic (see, e.g.~\cite{O'Sullivan:no_conjugate})

\begin{lemma}\label{lemma:Rafi_No_focal_version}
Let $g$ be a smooth Riemannian metric with no focal points on a closed surface $M$ with $\chi(M)<0$. Denote by $K_g^+$ the positive part of the Gaussian curvature on $M$. Suppose that, for some $C>0$, $\int_M K_g^+ dA_g \leq C$.

Let $\mathcal A$ be an annulus in $(M,g)$. Let $\gamma_0$ and $\gamma_1$ be its two boundary curves. Assume that $\gamma_0$ and $\gamma_1$ are both equidistant to a fixed geodesic, and that, for some $C_2\geq 0$, we have $\kappa(\gamma_0)\leq C_2$.
Then,
\begin{align*}
&\Mod_g(\mathcal A) = \frac{\dist_g(\gamma_0, \gamma_1)}{l_{g}(\gamma_0)} \text{ if } \kappa(\gamma_0)=0 \text{ and } \mathcal A \text{ is a flat annulus}, \\
&\Mod_g(\mathcal A)\geq \frac{1}{-2\pi\chi(M) +C +C_2}\ln\left(1 + (-2\pi\chi(M) +C +C_2)\frac{\dist_g(\gamma_0, \gamma_1)}{l_{g}(\gamma_0)}\right) \text{ otherwise,}
\end{align*}
where $\dist_g(\cdot, \cdot)$ is the distance function on $(M,g)$.
\end{lemma}

\begin{proof}
We use the same notation as in the proof of Lemma \ref{lemma: Rafi}, and will only add the modifications needed for this generalization.

As previously, we consider the level curves $\hat\gamma_r := \{p\in\mathcal A| \dist_g(p, \gamma_0) = r\}$. 

Since all the curves $\hat\gamma_r$ are equidistant to a fixed geodesic $\gamma$, they must foliate the annulus $\mathcal A$. Otherwise, we would have a point $x\in A$ with two distinct orthogonal projection onto the geodesic $\gamma$. This is impossible since $g$ has no focal points.

Then as before, we have that, for any curve $c$ between the boundaries of $\mathcal A$,
\begin{equation*}
l_{fg}(c) \geq l_{g}(\gamma_0)\int\limits_0^{\dist_g(\gamma_0, \gamma_1)}\frac{1}{l_{g}(\hat\gamma_r)}dr \quad \text{ and } \quad
\vol_{fg}(\mathcal A) =  l^2_{g}(\gamma_0)\int\limits_0^{\dist_g(\gamma_0, \gamma_1)}\frac{1}{l_{g}(\hat\gamma_r)}dr,
\end{equation*}
where $f$ is the function defined by $f(r):=\frac{l_{g}(\hat\gamma_0)}{l_{g}(\hat\gamma_r)}$.

Now, Gauss--Bonnet Theorem, applied to $M$ and $\mathcal{A}_r$, the annulus bounded by $\gamma_0$ and $\hat\gamma_r$, gives
\begin{equation*}
    \int_{M}K_{g} dA_{g}=2\pi\chi(M) \quad \text{ and } \quad
    \int_{\mathcal A_r}K_{g} dA_{g} = -\kappa(\hat\gamma_r)+\kappa(\gamma_0).
\end{equation*}
Thus, we obtain
\[
\kappa(\hat\gamma_r) \leq \kappa(\gamma_0) + \int_M -K_g dA_g + \int_M K_g^+ dA_g \leq -2\pi\chi(M) + C_2+C.
\]
And integration yields that $l_g(\hat\gamma_r) \leq (-2\pi\chi(M) +C_2+C)r +l_g(\gamma_0)$. Thus, as claimed, we obtain,
\begin{multline*}
    \Mod_g(\mathcal A) \geq\frac{\dist^2_{fg}(\gamma_0, \gamma_1)}{\vol_{fg}(\mathcal A)} \geq \int\limits_0^{\dist_g(\gamma_0,\gamma_1)} \frac{1}{l_{g}(\gamma_0)+(-2\pi\chi(M)+C_2+C)r}dr \\= \frac{1}{C_3} \ln \left( 1 + C_3 \frac{\dist_g(\gamma_0,\gamma_1)}{l_{g}(\gamma_0)}\right),
\end{multline*}
where $C_3= -2\pi\chi(M) +C_2+C$.
\end{proof}

We can now prove Theorem \ref{theorem: No focal points bound systole}. Since the proof follows exactly the same lines as Theorem \ref{theorem: bound systole}, we will use the same notations and only emphasize the changes that need to be made.

\begin{proof}[Proof of Theorem \ref{theorem: No focal points bound systole}]
As before, we let $\gamma$ be the shortest geodesic of $g$, $N_r$ its $r$-tubular neighborhood, and $Z_r$ the union of $N_r$ together with all the connected components of $M\smallsetminus N_r$ that are disks.

The only new difficulty now is that the boundary curves in $\partial Z_r$ may not be monotonically curved, so we will have to bound $\kappa(\partial Z_r)$ \emph{from below} (because of the choice of sign when defining the geodesic curvature) in order to be able to apply Lemma \ref{lemma:Rafi_No_focal_version}.

Thanks to Gauss--Bonnet Theorem, and the fact that $\chi(M)\leq \chi(Z_r) \leq 0$, we have
\begin{align*}
    \kappa(\partial Z_r) &= 2\pi \chi(Z_r) - \int_{Z_r} K_g dA_g \leq -2\pi\chi(M) + \int_M K_g^+ dA_g \leq  -2\pi\chi(M) + 2\pi-\eps,\\
    \kappa(\partial Z_r) &\geq 2\pi\chi(M) - \int_M K_g^+ dA_g \geq 2\pi\chi(M) - 2\pi + \eps.
\end{align*}

We let $K_1= -2\pi\chi(M) + 2\pi-\eps$. So $-K_1 \leq \kappa(\partial Z_r) \leq K_1$.

Then, as before, we obtain that
\begin{equation*}
l_g(\partial Z_r)-l_g(\partial Z_0) = \int\limits_0^r\kappa(\partial Z_\tau) d\tau - \sum\limits_{u\in I_r}c_u, \text{ and }
A_g(Z_r) - A_g(Z_0) = \int\limits_0^rl_g(\partial Z_{\tau}) d\tau + \sum\limits_{u\in I_r}A_g(D_u).
\end{equation*}

Now, Alexandrov's version of the isoperimetric inequality (see, e.g. \cite[section 2.2]{BuragoZalgaller}) implies that
\begin{equation*} 
A_g(D_u)\leq \frac{c_u^2}{2\left(2\pi - \int_{D_u}K_g^+ dA_g\right)} \leq \frac{c_u^2}{2\eps}.
\end{equation*}

\begin{remark}
Notice that this is the essential place where we need the total positive curvature to be strictly less than $2\pi$. Otherwise, one can shrink the systole by building a sequence of metrics on the surface such that all the area goes inside a disc. Then any curve that do not enter that disc will have length going to zero.
\end{remark}

The proof now follows exactly as in Theorem \ref{theorem: bound systole}, but with the appropriate changes of bounds. Indeed, we get, for any $r$,
\begin{equation*}
l_g(\partial Z_r)\leq K_1r+2l_g(\gamma), \text{ and } \sum\limits_{u\in I_r}c_u\leq K_1r+2l_g(\gamma),
\end{equation*}
Thus,
\begin{equation*}
    A_g(Z_r) \leq \frac{K_1r^2}{2} + 2l_g(\gamma) r + \frac{1}{2\eps} \left(K_1r+2l_g(\gamma) \right).
\end{equation*}

Now, we want to apply Lemma \ref{lemma:Rafi_No_focal_version} to each annuli in $Z_{r_{i+1}}\smallsetminus Z_{r_i}$. We denote by $\gamma_i \subset \partial Z_{r_i}$ and $\gamma_{i+1}\subset \partial Z_{r_{i+1}}$ the two boundary components of the annuli, and $\kappa(\gamma_i)$ for the total geodesic curvature, \emph{with respect to the annuli} in $Z_{r_{i+1}}\smallsetminus Z_{r_i}$. Then 
\[
\kappa(\gamma_i) \leq - \kappa(\partial Z_{r_{i+1}}) \leq K_1.
\]
Thus, Lemma \ref{lemma:Rafi_No_focal_version} gives
\[
r_{i+1}-r_i \leq \frac{e^{E\left(-2\pi\chi(M) + K_1 + 2\pi - \eps\right)}-1}{-2\pi\chi(M) + K_1 + 2\pi - \eps}(K_1r_i+2l_g(\gamma)),
\]
where $E= E(\sigma)$ is the supremum of the modulus (in the conformal class of $\sigma$) of all the annuli in $M$.
The same computations as in the proof of Theorem \ref{theorem: bound systole} then yield
$\sys(g) \geq \sqrt{\frac{A}{R}}$, for an appropriate $R$, depending only on $E$, $\chi(M)$ and $\eps$.
\end{proof}

We end the section by noticing that Corollary \ref{corollary: bound topological entropy} also extends to the no focal point setting, since Saboureau's result \cite[Theorem 1.2]{Sabourau} holds for any metric, and the topological entropy coincides with the volume entropy in the case of metrics with no focal points \cite{Katok}. Thus, we obtain

\begin{corollary}\label{corollary: bound topological entropy no focal}
Let $M$ be a closed surface of negative Euler characteristic and $\sigma$ be a hyperbolic metric on $M$. Let $A,\eps>0$. Then, there exists a positive constant $B = B(\sigma, A, \eps)$ such that, if $g$ is a Riemannian metric with no focal points in the conformal class of $\sigma$, of total area $A$, and 
\[
\int_g K_g^+ dA_g \leq 2\pi-\eps,
\]
then
\begin{equation*}
h_{\topol}(g)\leq B,
\end{equation*}
where $h_{\topol}(g)$ is the topological entropy of the geodesic flow on $M$ with respect to the metric $g$.
\end{corollary}

\section{Compactification of metrics in a fixed conformal class} \label{section_precompactness}


\begin{definition}\label{uniform metric}
A sequence of metrics $\{g_k\}$ converges to a metric $g$ on $M$ in the uniform metric topology if there are diffeomorphisms $\phi_{k}\colon M\rightarrow M$ such that the sequence $(\phi_{k}^*g_{k})$ converges to $g$ uniformly on $M$. 
\end{definition}

\begin{theorem}\label{precompactness}
The set of metrics in a fixed conformal class, with no focal points, total area $A$, and total positive curvature less than $2\pi-\eps$
is precompact in the uniform metric sense. Moreover, if a metric $g$ belongs to the limiting set, then $g$ is a metric with bounded integral curvature in the sense of Alexandrov (see \cite[Section 1.1]{Debin}).
\end{theorem}

\begin{proof}
Theorem~\ref{theorem: No focal points bound systole} together with \cite[Corollary 4]{Debin} show that the set of considered metrics is precompact in the uniform metric sense and the limiting metrics have bounded integral curvature. 
\end{proof}




\section{Thick-thin decomposition}\label{section:thick-thin}

A hyperbolic surface $(M,\sigma)$ can be decomposed into thick and thin parts (see \cite[Chapter D]{BenedettiPetronio}). The thin part has a simple topology because the components of it are homeomorphic to annuli. The thick part has a bounded geometry in the sense that the diameter and the injectivity radius of a component of the thick part are bounded below and above by a constant depending only on the topology of $M$. 

In this section we show that a thick component equipped with a non-positively curved metric in the conformal class of $\sigma$ and of fixed total area has a geometry comparable to the $\sigma$-geometry of that piece. The following theorem is an analogue of \cite[Theorem 1]{Rafi_thick_thin} in our setting.

\begin{theorem}\label{thick_thin}
Let $M$ be a closed surface of negative Euler characteristic and $\sigma$ be a hyperbolic metric on $M$. Let $A>0$. Denote by $Y$ a component of the thick part of $(M, \sigma)$.  Then, there exist positive constants $C_1, C_2$ depending only on $\sigma, A$ and $\chi(M)$ such that:

For any non-trivial, non-peripheral, piecewise-smooth simple closed curve $\alpha$ in $Y$ and any smooth non-positively curved metric $g$ conformally equivalent to $\sigma$ with total area $A$, we have 
\begin{equation}\label{length comparison}
C_1l_{\sigma}(\alpha_{\sigma})\leq l_g(\alpha_g)\leq C_2l_{\sigma}(\alpha_{\sigma}),
\end{equation}
where $\alpha_g$ is the $g$-geodesic representative of $\alpha$.
\end{theorem}

Notice that we are now back in the setting of non-positively curved metrics, as opposed to the more general ones we considered in Section \ref{section: no focal points}. This is because we will use some results from \cite{BarthelmeErchenko} that were only proved for non-positively curved metrics.

We will need the following lemma.

\begin{lemma}\label{Rafi_lemma}(Version of \cite[Lemma 5]{Rafi_thick_thin} in a fixed conformal class)
The setting is as in Theorem~\ref{thick_thin}. Let $\alpha$ and $\beta$ be two non-trivial non-peripheral piecewise-smooth simple closed curves in $Y$. Then, there exists a positive constant $D = D(\sigma, A, \chi(M))$ such that for any smooth non-positively curved metric $g$ conformally equivalent to $\sigma$ with total area $A$, we have
$$l_g(\alpha)l_g(\beta)\geq D \im(\alpha,\beta),$$
where $\im(\cdot, \cdot)$ is the intersection number.
\end{lemma}

\begin{proof}
The proof of Lemma 5 in \cite{Rafi_thick_thin} applies verbatim, just using the fact that for any smooth non-positively curved metric conformally equivalent to $\sigma$ with total area $A$ the $g$-size of $Y$ (see the introduction of \cite[Section 3]{Rafi_thick_thin}) is bounded below by $C$ thanks to our Theorem~\ref{theorem: bound systole}.
\end{proof}

\begin{proof}[Proof of Theorem~\ref{thick_thin}]
Let $\alpha$ be a non-trivial non-peripheral piecewise-smooth simple closed curve in $Y$. 

By \cite[Theorem A]{BarthelmeErchenko} there exists a constant $C_2=C_2(\sigma, A)$ such that 
\begin{equation*}
l_g(\alpha_{\sigma})\leq C_2l_{\sigma}(\alpha_{\sigma}).
\end{equation*}
Since $l_g(\alpha_g)\leq l_g(\alpha_{\sigma})$, we directly obtain the right hand side inequality in Equation \eqref{length comparison}.

Now we will prove the left hand side inequality in Equation \eqref{length comparison}.

Let $\mu$ be a short marking of $Y$. That is, $\mu$ is a collection of the following curves: First, $\mu$ contains all the non-trivial simple closed $\sigma$-geodesics in the $\sigma$-shortest pants decomposition of $Y$ (i.e., the sum of the $\sigma$-lengths of the cuffs of the pants is as small as possible). Then, for each such curves, we add to $\mu$ the transverse, non-trivial, non-peripheral simple closed (note that it could have endpoints on the boundary of $Y$) curve with the shortest $\sigma$-length.

Let $L_{\sigma}(\mu) = \sum\limits_{\beta\in\mu}l_{\sigma}(\mu)$ be the $\sigma$-length of $\mu$. Note that $L_{\sigma}(\mu)$ depends only on $\sigma$ and the topology of $M$. 

Then, by Lemma~\ref{Rafi_lemma} and \cite[Theorem A]{BarthelmeErchenko}, we obtain that
\begin{align*}
D\sum\limits_{\beta\in\mu}\im(\alpha_g,\beta)\leq \sum\limits_{\beta\in\mu}l_g(\alpha_g)l_g(\beta)\leq l_g(\alpha_g)\sum\limits_{\beta\in\mu}C_2l_{\sigma}(\beta) = C_2L_{\sigma}(\mu)l_g(\alpha_g).
\end{align*}
Finally, we have
\begin{equation*}
l_g(\alpha_g)\geq \frac{D}{C_2L_{\sigma}(\mu)}\sum\limits_{\beta\in\mu}\im(\alpha_g,\beta) = \frac{D}{C_2L_{\sigma}(\mu)}\sum\limits_{\beta\in\mu}\im(\alpha_{\sigma},\beta)\geq \frac{D}{C_2L_{\sigma}(\mu)}D_2 l_{\sigma}(\alpha_{\sigma}), 
\end{equation*}
where in the last inequality we used that there exists a positive constant $D_2$ that depends on $\sigma$ and the topology of $M$ such that $\sum\limits_{\beta\in\mu}\im(\alpha_{\sigma},\beta)\geq D_2l_{\sigma}(\alpha_{\sigma})$ (see the proof of \cite[Lemma 4.7]{Minsky_intersection}). As a result, we get the left inequality in \eqref{length comparison} with $C_1 = \frac{DD_2}{C_2L_{\sigma}(\mu)}$. 
\end{proof}

\section{Flexibility of the metric entropy}\label{section:metric entropy}

In this section we prove Conjecture 1.1 of \cite{BarthelmeErchenko}.

\begin{theorem}\label{flex_metric}
Let $M$ be a closed surface of negative Euler characteristic and $\sigma$ be a hyperbolic metric on $M$. Let $A>0$. Then,
$$\inf_{g\in [\sigma]^<_A}h_{\metr}(g)=0,$$
where $h_{\metr}(g)$ is the metric entropy with respect to the Liouville measure of the geodesic flow on $(M,g)$ and $ [\sigma]^<_A$ is a family of smooth negatively curved metrics conformally equivalent to $\sigma$ with total area $A$.
\end{theorem}

To prove Theorem~\ref{flex_metric}, we will need the following lemma. 

\begin{lemma}\label{smoothing_lemma}(Version of \cite[Lemma 1]{Ramos} for non-positively curved metrics)
Denote by $D$ the unit disk in $\mathbb C$. Let $g_0 = e^{2(a_0(z)+\beta\ln |z|)}|dz|^2$ be a cone metric on the punctured disk $D\setminus\{0\}$, where $\beta >0$ and $a_0(\cdot)$ is a smooth function on $D$, chosen so that the curvature of $g_0$, $K_{g_0}(\cdot)$, is non-positive.

Then there exists a decreasing sequence of smooth metrics $g_k = e^{2u_k}|dz|^2$ on $D$ such that:
\begin{enumerate}[label=(\roman*)]
\item $g_k = g_0$ on $D\setminus D_{\frac{1}{k}}$, where $D_{\frac{1}{k}}$ is a disk of radius $\frac{1}{k}$ (for the Euclidean metric on $\mathbb C$) centered at $0$;\label{lemma (1)}
\item $u_k\geq u_0$ on $D\setminus \{0\}$;\label{lemma (2)}
\item $\inf\limits_{D_{\frac{1}{k}}}u_k\rightarrow -\infty$ as $k\rightarrow +\infty$;\label{lemma (3)}
\item The Gaussian curvature function $K_{g_k}(\cdot)$ on $(D,g_k)$ satisfies $K_{g_k}(z)\leq 0$ for any $z\in D$. \label{lemma (4)}
\end{enumerate}

\begin{remark}
Lemma~\ref{smoothing_lemma} can be of independent interest. In particular, it can be used to define the Ricci flow in the spirit of \cite[Theorem 3.1]{Ramos} on surfaces of non-positive curvature everywhere except for finitely many points with conical singularities of angles larger than $2\pi$. This Ricci flow will smoothen conical points while preserving non-positive curvature. 
\end{remark}

\begin{proof}[Proof of Lemma~\ref{smoothing_lemma}]
The proof follows the ideas of \cite[Lemma 1]{Ramos}.

Let $(r,\theta)$ be polar coordinates on $D$. Consider the conformal factor $$u_0(r,\theta)=a_0(r,\theta)+\beta\ln r$$ of the metric $g_0$. Notice that $u_0(r,\theta)$ tends to $-\infty$ as $r\rightarrow 0$.

For each natural number $k>\sqrt{\frac{\beta+2}{\beta}}$ we define $v_k(r)=C_k-\ln(1-r^2)$, where $C_k = \ln (1-\frac{1}{k^2})-1+\beta\ln\frac{1}{k}+\min\limits_{D}a_0(r,\theta)$. In particular, $v_k(0)=C_k\rightarrow -\infty$ as $k\rightarrow +\infty$ and $u_0(r,\theta)-v_k(r)\geq 1$ for any $\theta$ and $r\in[\frac{1}{k},\sqrt{\frac{\beta}{\beta+2}}]$. Moreover, the metric $e^{2v_k}|dz|^2$ on $D_{\frac{1}{k}}$ has constant negative curvature $-4e^{-2C_k}\rightarrow -\infty$ as $k\rightarrow +\infty$. 

Choose a smooth function $\psi\colon \mathbb R\rightarrow \mathbb R$ such that
\begin{enumerate}
\item $\psi(s) = -s$ for $s\leq -1$;
\item $\psi(s) = 0$ for $s\geq 1$;
\item $-1\leq \psi'(s)\leq 0$ and $\psi''(s)\geq 0$ for any $s$.
\end{enumerate}

Define a smooth function
\begin{equation}
u_k = \begin{cases}
\psi(u_0-v_k)+u_0 &\text{if } 0 \leq r\leq \frac{1}{2k}+\frac{1}{2}\sqrt{\frac{\beta}{\beta+2}},\\
u_0 &\text{otherwise.}
 \end{cases}
\end{equation}
In particular \ref{lemma (2)} holds because $\psi(s)\geq 0$ for every $s\in R$.

The function $u_k$ is smooth because $u_0$ is smooth outside of any neighborhood of $r=0$, $u_k = u_0$ for $r\in[\frac{1}{k},\frac{1}{2k}+\frac{1}{2}\sqrt{\frac{\beta}{\beta+2}}]$, and $u_k=v_k$ in some neighborhood of $r=0$. In particular, \ref{lemma (1)} and \ref{lemma (3)} in Lemma~\ref{smoothing_lemma} holds.

Moreover, we have
\begin{enumerate}[label=(\alph*)]
\item Let $D'$ be the subset of $D$ such that $u_0(z)\leq v_k(z)-1$ for $z\in D'$. Then, $u_k=v_k$ and $K_{g_k}=-4e^{-2C_k}<0$ on $D'$.
\item Let $D''$ be the subset of $D$ such that $u_0(z)\geq v_k(z)+1$ for $z\in D''$. Then, $u_k=u_0$ and $K_{g_k}\leq 0$ on $D''$.
\item Let $D'''$ be the subset of $D$ such that $v_k(z)-1<u_0(z)< v_k(z)+1$ for $z\in D'''$. We need to check that $K_{g_k}(z)\leq 0$ for $z\in D''$.

Recall that $K_{g_k} = -e^{-2u_k}\Delta u_k$. Therefore, $K_{g_k}\leq 0$ if and only if $\Delta u_k\geq 0$. In particular, $\Delta u_0\geq 0$ on $D'''$ as $K_{g_0}\leq 0$ on $D\setminus \{0\}$

Using the conditions on $\psi$, we have the following on $D'''$:
\begin{align*}
\Delta u_k &= \psi''(u_0-v_k)|\nabla (u_0-v_k)|^2+\psi'(u_0-v_k)\Delta (u_0-v_k)+\Delta u_0 \\&\geq -\psi'(u_0-v_k)\Delta v_k = -\psi'(u_0-v_k)\frac{4}{(1-r^2)^2}\geq 0.
\end{align*}
\end{enumerate}

Therefore, \ref{lemma (4)} in Lemma~\ref{smoothing_lemma} holds.
\end{proof}
\end{lemma}

\begin{proof}[Proof of Theorem~\ref{flex_metric}]

Pick a point $p$ on $M$. Then, a result of \cite[Section 5]{Troyanov} states that there exists a unique metric $g$, conformally equivalent to $\sigma$, of total area $A$, and of zero curvature everywhere except at the point $p$ where it has a conical singularity of angle $\alpha = 2\pi(1-\chi(M))$. In particular, $p$ admits an open neighborhood $\mathcal U$ and there exists a diffeomorphism from $\mathcal U\setminus\{p\}$ to $D\setminus \{0\}$ such that the metric $g$ in the coordinates of $D\setminus \{0\}$ have the following expression
$$g=(\beta+1)^2r^{2\beta}|dz|^2,$$
where $\beta = \frac{\alpha}{2\pi}-1>0.$

Denote by $g_k= e^{2u_k}|dz|^2$ the family of smooth metrics given by Lemma~\ref{smoothing_lemma} applied to the metric $g$. The $g$-radius of the disk $D_{\frac{1}{k}}$ of radius $r=1/k$ centered at $0$ is equal to $1/k^{\beta+1}$ and has $g$-area $\pi(\beta+1)/k^{2\beta+2}$. In particular, the $g$-radius and $g$-area of $D_{\frac{1}{k}}$ tends to $0$ as $k\rightarrow +\infty$. 

Using the notations of the proof of Lemma~\ref{smoothing_lemma}, we have 
\begin{equation*}
u_0=u_0(r) = \ln(\beta+1)+\beta\ln r \text{ and } v_k(r) = \ln\left(1-\frac{1}{k^2}\right)-1+\beta\ln\frac{1}{k}+\ln(\beta+1)-\ln(1-r^2).
\end{equation*}

In particular, $u_0(1/k)-v_k(1/k)=1$ and $u_0(r)-v_k(r)$ increase when $r\in\left(0,\sqrt{\beta/(\beta+2)} \right)$ and decrease when $r\in\left(\sqrt{\beta/(\beta+2)},1\right)$. Moreover, $u_0(r)-v_k(r)\leq -1$ and $u_k=v_k$ if $r\in[0,\frac{1}{k}e^{-2/\beta}(1-\frac{1}{k^2})^{1/\beta}]$. Therefore, there exists $C>0$ and $K>0$ such that, for any $k>K$, and any $z\in D_{\frac{1}{k}}$, the curvature satisfies $K_{g_k}(z) \geq -Ck^{2\beta}$.

Finally, applying the arguments of \cite[Section 3.3]{ErchenkoKatok}, we obtain $$h_{\metr}(g_{k})\rightarrow 0 \text{ as } k\rightarrow \infty.$$
In particular, $\inf\limits_{g\in[\sigma]^\leq_A}h_{\metr}(g)=0$.

Let $\eps_0>0$ be a sufficiently small number. Then, by \cite[Theorem A]{Troyanov_general}, for any $0\leq\eps<\eps_0$ there exists a metric $g_{\eps}$ of constant curvature $-\eps$ everywhere except a point where it has the conical singularity with angle larger than $2\pi$ which has the total area $A$ and is conformally equivalent to $\sigma$.
Following the same argument as above by starting with metric $g_\eps$, we obtain 
\begin{equation*}
\inf\limits_{g\in[\sigma]^<_A}h_{\metr}(g)=0. \qedhere
\end{equation*}
\end{proof}
\section{Further questions}\label{section:questions}

In this section, $M$ is still a closed surface of negative Euler characteristic $\chi(M)$ and $\sigma$ is a hyperbolic metric on $M$.


\subsection{Possible values of entropies in a fixed conformal class}

By \cite[Theorem B]{Katok}, we know that for any smooth negatively curved Riemannian metric $g$ on $M$ which is not a metric of constant curvature, we have the following inequalities for the metric entropy $h_{\metr}(g)$ with respect to the Liouville measure and the topological entropy $h_{\topol}(g)$ of the geodesic flow on $(M,g)$

\begin{equation}\label{constraint_entropies}
0<h_{\metr}(g)<\left(\frac{-2\pi\chi(M)}{A}\right)^{\frac{1}{2}}<h_{\topol}(g).
\end{equation}

In \cite{ErchenkoKatok}, A.~Katok and the second author proved that any two pair of reals satisfying to the above inequality are realized as a pair $(h_{\metr}(g), h_{\topol}(g))$ of a negatively curved metric (with fixed total area $A$).

On the other hand, Theorem~\ref{flex_metric} and Corollary~\ref{corollary: bound topological entropy}, show that, when one fixes the conformal class, then the metric entropy can be arbitrary close to $0$ whereas the topological entropy is bounded above.

Thus, it is natural to try to understand the possible pairs $(h_{\metr}(g), h_{\topol}(g))$ where $g\in[\sigma]^<_A$.
%
\begin{question}\label{q: max_top}
What is the graph of the function $$H^{\topol}_{\sigma}(x) := \sup\{h_{\topol}(g)\,|\,g\in[\sigma]^<_A,\, h_{\metr}(g)=x \}$$
where $x\in\left(0,\left(\frac{-2\pi\chi(M)}{A}\right)^{\frac{1}{2}}\right]?$
\end{question}

While it seems hard to answer Question~\ref{q: max_top}, a good first step would be to answer the following questions.

\begin{question}
For any $x\in\left(0,\left(\frac{-2\pi\chi(M)}{A}\right)^{\frac{1}{2}}\right]$, does there exists $g\in[\sigma]^<_A$ (or $g\in[\sigma]^\leq_A$) such that $$h_{\metr}(g)=x \quad\text{ and }\quad h_{\topol}(g) = H^{\topol}_{\sigma}(x)?$$
\end{question}

\begin{question}\label{max_top_on_flat}
If $\lim\limits_{x\rightarrow 0+}H_{\sigma}^{\topol}(x)$ exists, what is its value in terms of $\sigma$?
\end{question}

Note that Question~\ref{max_top_on_flat} basically asks what is the supremum of the topological entropy of the geodesic flow (``properly" defined) on singular flat metrics that are conformally equivalent to $\sigma$ and have total area $A$.

While we do not know the answers to the above questions, we expect that the set of possible pairs $(h_{\metr}(g), h_{\topol}(g))$ where $g\in[\sigma]^<_A$ looks like the shaded region on Figure~\ref{figure: entropies}. Indeed, considering \cite[Theorem 5.1]{BarthelmeErchenko} and \cite[Section 2]{ErchenkoKatok}, one sees that to increase topological entropy one needs to shrink a non-trivial simple closed curve. Now, to preserve negative curvature we need to modify the metric on some neighborhood of that curve whose size, most likely, will depend on the conformal class. Therefore, we do not expect that, in a fixed conformal class it is possible to increase topological entropy while having the metric entropy arbitrary close to $\left(\frac{-2\pi\chi(M)}{A}\right)^{\frac{1}{2}}$ (i.e., we expect a gap between the shaded domain and the vertical line in Figure~\ref{figure: entropies}).

Moreover, given the construction in \cite[section 3]{ErchenkoKatok} and corollary ~\ref{corollary: bound topological entropy}, we expect that, for any hyperbolic metric $\sigma$, the limit $\lim\limits_{x\rightarrow 0+}H_{\sigma}^{\topol}(x)$ exists. Notice that this limit will go to infinity as $\sigma$ leaves every compact of the Teichm\"uller space.


\begin{figure}
\centering
  \begin{tikzpicture}[scale=1.2]
\draw[dashed] (1,1) -- (1,4);
 \draw[->] (-0.5,0) -- (3,0) node at (2,-0.3) {Metric entropy}; 
    \draw (0,-0.5) -- (0,1); 
    \draw[->] (0,3) -- (0,4) 
    node[rotate = 90] at (-0.3, 3) {Topological entropy};
    \draw[yellow, fill=yellow, line width = 1.3] (0,3)--(0,1)--(1,1)--plot [samples=200, domain=0:1]  (\x,-\x*\x*\x*\x*\x*\x*\x*\x-\x*\x*\x*\x*\x*\x+3);
    \draw[dashed, line width = 1.3] (0,1)--(1,1);
    \draw[dashed, line width = 1.3] (0,1)--(0,3);
		\draw [dashed, line width = 1.3] plot [samples=200, domain=0:1]  (\x,-\x*\x*\x*\x*\x*\x*\x*\x-\x*\x*\x*\x*\x*\x+3);
    \node(x) at (1, 1) {$\bullet$};
\node[align=right](y) at (2.8, 1) { Metric of constant \\  negative curvature};
\draw[->] (y) -- (x);
\end{tikzpicture} 
 \caption{Conjectural possible values of entropies in a fixed conformal class.}
      \label{figure: entropies}
\end{figure}
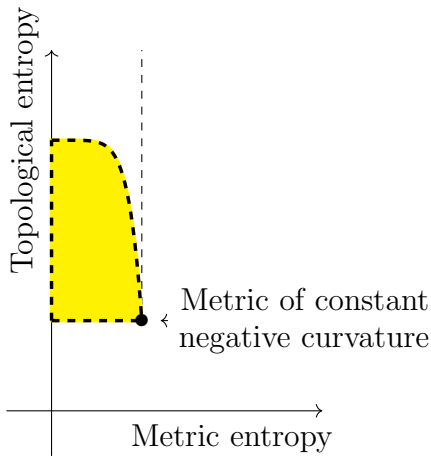

We also expect that there should be negatively curved metrics $g$ in any fixed conformal class such that $(h_{\metr}(g), h_{\topol}(g))$ is any point of the (admissible) neighborhood of $\left(0,\sqrt{-2\pi\chi(M)/A}\right)$.

Indeed, in \cite{Kerckhoff}, Kerckhoff proved that for any Riemannian metrics $g_1$ and $g_2$ the Teichm\"uller distance $d_{Teich}(g_1,g_2)$ between their conformal classes is equal to 
\begin{equation}\label{Teichmuller distance}
d_{Teich}(g_1,g_2) = \frac{1}{2}\log\left(\sup\limits_{\gamma}\frac{E_{g_1}(\gamma)}{E_{g_2}(\gamma)}\right),
\end{equation}
where $\gamma$ ranges over all non-trivial simple closed curves (see Definition~\ref{extremal_length} for $E_{g}(\gamma)$). Thus, $C^0$-closeness of Riemannian metrics implies closeness of their conformal classes in the Teichm\"uller space. Therefore, the examples built in \cite[Section 3.1]{ErchenkoKatok} such that $(h_{\metr}(g), h_{\topol}(g))$ is in the neighborhood of $\left(0,\sqrt{-2\pi\chi(M)/A}\right)$ belong to conformal classes not far from the conformal class of $\sigma$. It is thus likely that one can make similar examples in a fixed conformal class.

\subsection{What is in the compactification of $[\sigma]^\leq_A$?}

By Theorem~\ref{precompactness}, the set of metrics $[\sigma]^{\leq}_A$ is precompact in the uniform metric sense. Moreover, $g$ is a metric of bounded integral curvature. What seems not to be known is how ``singular'' the metric is.

\begin{question}
What are the properties of a metric which is the limit of a sequence of metrics in $[\sigma]^\leq_A$? 
\end{question}

\subsection{Flexibility beyond two entropies}

There are other interesting and important intrinsic characteristics of the geodesic flow on negatively curved surfaces beside $h_{\metr}(\cdot)$ and $h_{\topol}(\cdot)$. Let $h_{\harm}(g)$ be the entropy of the geodesic flow on $(M,g)$ with respect to the harmonic invariant measure. Denote by $\lambda_{\max}(g)$ the positive Lyapunov exponent with respect to the measure of maximal entropy.

The following inequalities hold for any negatively curved metrics with fixed total area $A$ (see \cite{Ruelle}, \cite{Katok}, and \cite{Ledrappier}).
\begin{equation}\label{some_ineq}
h_{\metr}(g)\leq  \left(\frac{-2\pi\chi(M)}{A}\right)^{\frac{1}{2}}  \leq h_{\harm}(g) \leq h_{\topol}(g) \leq \lambda_{\max}(g).
\end{equation}
Moreover, if any of the inequalities above is an equality, then all the other also are equalities and the metric $g$ has constant curvature.

By Corollary~\ref{corollary: bound topological entropy}, we know that there exists a uniform upper bound for $h_{\topol}(\cdot)$ on $[\sigma]^<_A$. Therefore, the next natural question is the following.

\begin{question}
Does there exist a uniform upper bound for $\lambda_{\max}(\cdot)$ on $[\sigma]^<_A$?
\end{question}

In terms of the study of the flexibility properties of geometric and dynamical data, then a very general question is
\begin{question}
What four-tuples of positive numbers satisfying inequalities \eqref{some_ineq} are realizable on $[\sigma]^<_A$?
\end{question}

\bibliography{bibliography}{}
\bibliographystyle{alpha}

\end{document}